%% file: robos.tex
\documentclass[12pt]{amsart}

\pdfoutput=1
\usepackage[mathcal]{euscript}
\usepackage{mathrsfs}
\usepackage{amsfonts,amssymb,amsbsy,amsmath}
\usepackage{graphicx,color}
\usepackage{epsfig}
\usepackage[all,cmtip]{xy}
\usepackage{verbatim}
\usepackage{setspace}
\usepackage{comment}
\usepackage{soul}

\usepackage[normalem]{ulem}

\makeatletter
\let\uppercasenonmath\@gobble% disables title uppercase
\makeatother

\newcommand{\disp}{\displaystyle}
\newcommand{\ba}{\begin{array}}
\newcommand{\ea}{\end{array}}
\newcommand{\bea}{\begin{eqnarray}}
\newcommand{\eea}{\end{eqnarray}}

\newcommand{\pr}{\mbox{pr}}
\newcommand{\diam}{\mbox{diam}}
\newcommand{\ra}{\rightarrow}

\newcommand{\R}{\mathbb{R}}

\newcommand{\Z}{\mathbb{Z}}

\renewcommand{\phi}{\varphi}

\theoremstyle{plain}
\newtheorem{theorem}{Theorem}
\newtheorem{lemma}[theorem]{Lemma}
\newtheorem{coro}[theorem]{Corollary}
\newtheorem{prop}[theorem]{Proposition}
\theoremstyle{remark}

\theoremstyle{definition}

\begin{document}

\title{Sequential gradient dynamics in real analytic Morse systems}

\author{H. I\c{s}{\i}l Bozma}
\address{Bo\u{g}az\.{\i}\c{c}\.{\i} \"{U}n\.{\i}vers\.{\i}tes\.{\i}, Department of Electrical and Electronics Engineering, \.Istanbul, Turkey}
\email{bozma@boun.edu.tr}

\author{Fer\.{\i}t \"{O}zt\"{u}rk}
\address{Bo\u{g}az\.{\i}\c{c}\.{\i} \"{U}n\.{\i}vers\.{\i}tes\.{\i}, Department of Mathematics, \.Istanbul, Turkey}
\email{ferit.ozturk@boun.edu.tr}


\begin{abstract}
Let $\Omega\subset\R^{M}$ be a compact connected $M$-dimensional real analytic  domain with boundary and $\phi$ be a primal navigation function; i.e.  
a real analytic  Morse function on $\Omega$ with a unique minimum 
and  with minus gradient vector field $G$ of  $\phi$ on ${\partial \Omega}$ 
pointed inwards along each coordinate.
Related to a robotics problem, we define a sequential hybrid process on $\Omega$ for  $G$ starting
from any initial point $q_0$ in the interior  of $\Omega$ as follows: at each step we restrict ourselves to an affine subspace where a collection of coordinates are fixed and  allow the other coordinates change along an integral curve of the projection of $G$ onto the subspace.  We prove that provided each coordinate appears infinitely many times  in the coordinate choices during the process, the process converges to a critical point of $\phi$. That critical point is the unique minimum for a dense subset of primal navigation
functions. We also present an upper bound for the total length of the trajectories close to a critical point. 
 \end{abstract}

\maketitle

\section{Introduction}

We consider the $d$-dimensional closed unit  ball $D\subset \R^d$ ($d\geq 1$) and $N$ objects (point particles or small disks with various radii) that we call {\it robots}. Each robot has an
initial position and a target position in the interior of $D$. A problem stated in e.g.  \cite{far}, \cite{fco}, \cite{boz}, \cite{lat}, is to navigate the robots from their initial to the target positions without overlapping each other nor the boundary of $D$.
The configuration space of nonoverlapping positions of the robots considered as point particles is a noncompact  manifold in $\R^{dN}$,  and is connected except for $d=1, N>1$ \cite{bozkodrob2001}. Allowing nonzero radii for robots, the configuration space of nonoverlapping positions (touching allowed) is a compact manifold-with-boundary $\Omega$ in $D^N\subset \R^{dN}$. Connectedness depends on the relative values of the radii; in particular if the radii of the robots are sufficiently small then the configuration space is connected (see e.g. \cite{bar}). In the sequel we assume that $\Omega$ is connected.  

The initial and target positions determine two points, say $q_0$ and $p$ respectively, in the interior of 
$\Omega$. Let $\tilde{\Omega}$ be a small, open neighborhood of $\Omega$.
Suppose there is a $C^2$  function $\phi:\tilde{\Omega}\ra\R$ such that \\

(i) $\phi$ is Morse on $\Omega$; 

(ii) $-\nabla \phi$ is inwards on ${\partial \Omega}$;

(iii) $\phi$ has a unique minimum at $p$; there $\phi(p)=0$. \\

Note that the condition {that the  domain of $\phi$ contains $\Omega \subset \tilde{\Omega}$} properly guarantees 
that the derivative of $\phi$ on $\partial \Omega$ is defined. 
Now it is well known through Morse theory that the integral curve $\gamma_{q_0}$ of the vector field $-\nabla\phi$ through $q_0\in\Omega$ converges to a point $q\in\Omega$ and provided that $q_0$ is not chosen in some set of measure zero, we have $q=p$ (see e.g. \cite{mil}). More precisely, the measure zero set is the complement of the stable manifold $W^s_p$ of $p$ in $\Omega$.
This solves the navigation problem of the $N$ robots; they move in $D$ simultaneously along trajectories determined by $\gamma_{q_0}$ and each converges to its target point provided that $q_0\in W^s_p$.
In \cite{boz}, such a Morse function has been constructed using the basic idea that each robot will repel each other and  will be attracted by its target point. {Furthermore the function has been constructed real analytic}. 

In the present article, we ask a related question: suppose that the robots are allowed to move one at a time cyclically rather than simultaneously. Can the same function $\phi$ and its minus gradient still  be used  to solve this navigation problem? Let us make precise what we mean.
  
Let $\phi$ be a real analytic function satisfying (i)-(iii) above on a compact connected real analytic domain with boundary
(i.e. analytically diffeomorphic to a closed ball)  $\Omega$  in $\R^{dN}$ 
and $G= -\nabla\phi$. By definition the vector field $G$ is  real analytic.
% and the single minimum of $\phi$ is attained at a point $q$ in int($\Omega$), the interior of $\Omega$. 
We will denote the points of $\R^{dN}$ by 
$x=(x_1,\ldots,x_N)=((y_1,\ldots,y_d),\ldots,(y_{dN-d+1},\ldots,y_{dN}))$ with $x_m=(y_{dm-d+1},\ldots,y_{dm})\in\R^d$. (Now and below,  $x_m$ is a $d$-tuple and $y_j\in\R$; $m$  always denotes an index in $1,\ldots,N$ and $j$  an index in $J=\{1,\ldots,dN\}$.) 

Let $\mathbf{n}$ be a subset of $J$. 
For a point $a\in\Omega$, we let
$$R^{\mathbf{n}}_a\doteq\{(y_1,\ldots,y_{dN})|y_{t}=a_t \mbox{ for all } t\in J-\mathbf{n}\}\cap\Omega$$
and call it  the {\it slice of type $\mathbf{n}$}  through $a$.
To keep the notation simple we will denote the particular slice through $a$ of type $\mathbf{n}=\{dm-d+1,\ldots, dm\}$ by $R^m_a$ and we will call it a slice of type $m$.

Let us consider an initial point $q_0\in\mbox{int}(\Omega)$ and the vector field $G^1_{q_0}$ over $R^1_{q_0}$
defined for each $x\in R^1_{q_0}$ by $G^1_{q_0}(x)=\pr_1(G(x))=-\nabla(\phi|_{R^1_{q_0}})$
Here  $\pr_i$ is the projection map onto the $i$-th $d$-tuple of coordinates. Observe that  $G^1_{q_0}$ is a real analytic gradient vector field. 
In the sequel we will drop the subscript and write $G^1$ in case no confusion arises.
Now, \L ojasiewicz's theorem  states that any integral curve
of the gradient flow of a real analytic function that has a nonempty $\omega$-limit set has a unique
limit \cite{loj}. Thus in our case  if the orbit of the vector field
$G^1_{q_0}$ through $q_0$ stays in $R^1_{q_0}$, it converges to a unique point, say $q_1\in R^1_{q_0}$, which is a singular point of the vector field $G^1_{q_0}$. 
Let us denote the closure of that orbit by $\gamma_1$, which starts from $q_0$ and ends at $q_1$.

At this point, we will have an extra condition to guarantee that the orbits stay in $\Omega$.
For each $a\in\partial\Omega$ let $f_a\in C^{\omega}(\R^{dN},\R)$ be a defining analytic function for $\partial \Omega$ in an open subset $U_a$ containing $a$. Suppose that $\nabla f_a$ is outward.   Then we impose the following condition which is
stronger than condition (ii) above: \\

(ii)$'$ For every $a\in\partial \Omega$ and $x\in U_a\cap\partial \Omega$, the angle
between $\nabla\phi(x)$ and $\nabla f_a(x)$ is sufficiently small so that for every
%For every $a\in\partial \Omega$, $x\in U_a\cap\partial \Omega$ and 
$j\in J$, the $j$-th component of $\nabla\phi(x)$ and $\nabla f_a(x)$ have the same sign (or are zero). \\

This condition is equivalent to having the projected vector field 
$-\nabla(\phi|_{R^{\mathbf{n}}_a})$ inwards (or zero) at the boundary of every slice 
$R^{\mathbf{n}}_a$, for each $a\in \Omega$ and $\mathbf{n}\subset J$.
Observe that  (ii)$'$ is satisfied if, for example, the angle between $\nabla \phi$ and 
$\nabla f_a$ is identically zero on $\partial \Omega$. If, furthermore, $\partial \Omega$
is given as a level set of a unique real analytic function,  this situation is
nothing but the admissibility of a  navigation function referred in e.g. \cite{ko1}, \cite{ko2}, \cite{boz}.

Now recursively, at the $k$-th step, let $m_k=(k \mod N)+1$. We consider the orbit through the point $q_{k-1}$ of the vector field $G^{m_k}_{q_{k-1}}$
over the slice $R_{q_{k-1}}^{m_k}$. Again by Lojasiewicz's theorem and thanks to the condition (ii)$'$, the orbit converges to a unique singular point, say, $q_k$. We denote the closure of that orbit by $\gamma_k$. Note that $R_{q_{k}}^{m_k}=R_{q_{k-1}}^{m_k}$. We call $q_k$'s the stationary points of the process. We will say that 
the point $q_k$ is of type $m_k$; i.e. type-$m$ points are the limits of integral curves on type-$m$ slices. This  process producing the sequence $(q_k)_{k=0}^\infty$ of points will be called
a {\em sequential gradient dynamical process} associated to the slices $R^m_x$.
Note that there are infinitely many points of the sequence $(q_k)_{k=0}^\infty$ of type $m$ for each $m\in\{1,2,\ldots,N\}$. 

We deal with the following question: For what values of $q_0$ does the sequence $(q_k)$ converge to the unique minimum $p$? For other values  of $q_0$ does  $(q_k)$ converge?
Our purpose in this article is to prove the following theorem that answers these two questions.
To state that, we will need a definition: we call a real analytic  function satisfying the conditions (i), (ii)$'$, (iii) above a {\it  primal navigation function}.

\begin{theorem}
\label{ana}
Let $\Omega\subset \R^{dN}$ be a compact connected real analytic domain with boundary;
$\phi:\Omega\ra \R$ 
be a  primal  navigation function with the unique minimum at $p\in\mbox{int}(\Omega)$ and 
$q_0\in\mbox{int}(\Omega)$ be an arbitrary initial point. Then the sequence $(q_k)_{k=0}^\infty$ 
of the sequential gradient dynamical process converges to a critical point of $\phi$.
%and $\phi(p)=0$.  

Moreover, there is a primal  navigation function $\psi$ on $\Omega$ with the same minimum $p$
such that $\psi$ is arbitrarily $C^{\omega}$-close to $\phi$, 
%the critical points of $\psi$ are in one to one correspondence with, arbitrarily close to and of the same index  as those of $\phi$ 
and that the process for $\psi$ converges to $p$ for all initial points $q_0$ except when $q_0$ is a saddle point.

In any case of convergence above, it is possible to have  $(q_k)$ ultimately constant or there may be infinitely many nontrivial steps of the process.
\end{theorem}

{The proofs of the first two claims are given in }Section~\ref{pruf} as
Lemma~\ref{bir}, Corollary~\ref{iki} and Lemma~\ref{uc}. 
They depend heavily on $\phi$'s being real  analytic and Morse.
The main tool in the proofs is the \L ojasiewicz inequality \cite{loj0}: 
For $q$  a critical point of  the real analytic function $\phi$ and  $\phi(q)=0$,  there are  a neighborhood $U$ of $q$, $c>0$ and $\mu\in[0,1)$ satisfying for all $x\in U$, 
\begin{equation}
|\nabla\phi(x)|> c|\phi(x)|^{\mu}.
\label{oja}
\end{equation}
We will call the triple $(U,c,\mu)$ a \L ojasiewicz neighborhood. Being Morse  will be used only in Lemma~\ref{uc}. 
The examples for the last claim of the theorem are presented in Section~\ref{ornek}.

The proof of Theorem~\ref{ana} is valid for a more general claim regarding a more general
set-up for the process. 
Consider a sequence $(\mathbf{n}_i)_{i=1}^{\infty}$ of subsets of $\{1,\ldots,M\}$ ($M\in\Z^+$) and  the  slices  $R^{\mathbf{n}_i}_a\subset \R^{M}$ through $a\in\Omega$ of type 
$\mathbf{n}_i$
%  , i.e. $R^{\mathbf{n}_i}_x$ is the $|\mathbf{n}_i|$-dimensional affine subspace through $x$ where the coordinate $y_j$ is fixed if and only if $j\not\in \mathbf{n}_i$.
We define the process associated to the sequence $(\mathbf{n}_i)$ as follows: at the $i$-th step the dynamics occurs in the slice $R^{\mathbf{n}_i}_{q_{i-1}}$ and converges to the point $q_i$. 
Then we have the following
\begin{theorem} 
$\phi:\Omega\ra \R$, $p$, $q_0$ and $(\mathbf{n}_i)_{i=1}^{\infty}$ be as above. 
Suppose each index $j$ ($1\leq j\leq M$)  appears infinitely many times  as elements of the sets of the sequence $(\mathbf{n}_i)$. Then the sequence $(q_k)_{k=0}^\infty$ 
of the sequential gradient dynamics associated to  the sequence $(\mathbf{n}_i)$ 
converges to a critical point of $\phi$. 

Moreover, there is a primal  navigation function $\psi$ on $\Omega$ with the same minimum $p$
such that $\psi$ is arbitrarily $C^{\omega}$-close to $\phi$, 
and that the process for $\psi$ converges to $p$ for all initial points $q_0$ except when $q_0$ is a saddle point.
\label{ekstra}
\end{theorem}
The proof of Theorem~\ref{ekstra} follows closely that of Theorem~\ref{ana}, both of which we will present below.

\section{Proof of the theorems}
\label{pruf}
Before stating and proving the claims, we will need several definitions and observations. 

First, for the sake of simplicity of the notation, let us restrict ourselves (except in Lemma~\ref{uc})  to the case $d=2$. ($d=1$ oversimplifies the problem \cite{bozkodrob2001}.) Whatever we present and claim below applies similarly for other dimensions as well.  Now, let $\Gamma_j \subset \Omega$ denote the  set of points at which the vector field $G$ has the $j$-th component zero; i.e. 
$$\Gamma_{j} = \{x\in\Omega \, | \, \phi_j (x) = 0\}$$
where {$\phi_j (x) = \frac{\partial \phi}{\partial y_j}(x) $}.
Observe that every critical point of $\phi$ lies in $\Gamma_{j}$ for each $j=1,\ldots,2N$.  Let $z$ be a nondegenerate critical point {of $\phi$}. Since the Hessian of $\phi$ at $z$ is  nondegenerate by assumption, $\nabla \phi_j(z)$ is nonzero and $\nabla \phi_j(z)$ and $\nabla \phi_i(z)$ are linearly independent for every pair of distinct $j$ and $i$. Hence we conclude that  around $z$ each $\Gamma_j$ and $\Gamma_{j,i} := \Gamma_j\cap \Gamma_i$ are smooth  real analytic submanifolds
  of $\Omega$ of codim 1 and 2 respectively. Furthermore, again thanks to the nondegeneracy assumption, the tangent spaces of $\Gamma_j$'s at $z$ constitute a central hyperplane arrangement of coordinate planes.
Therefore, the intersection of any collection consisting of $n$ of $\Gamma_j$'s is a smooth real analytic submanifold around $z$ as well, of codim $n$. 

Now let us investigate the process more closely in case $\phi$ is a real analytic Morse function. The point $q_k$ of type $m_k$ which the process converges to at $k$-th step lies in the codim-2 (in general codim-$d$) submanifold $\Gamma_{2m_k-1,2m_k}$; in other words, $k$-th step converges to a point in $\Gamma_{2m_k-1,2m_k} \cap R^{m_k}_{q_k}$. At that point $G$ is perpendicular to the slice. We may view the process as a chase within the subspaces $\Gamma_m$'s, which are submanifolds of codim~2 near the critical points. So, for example, in a neighborhood of some $q\in\Omega$ if $\Gamma_m\cap \phi^{-1}(\R_{\geq \phi(q)})$ is empty for some $m$ then the process cannot converge to $q$. 

Now we are ready for the claims and proofs. In the first two statements, we assume that $\phi$ is a real analytic function, not necessarily Morse. Observe that the process is still well-defined in the absence of being Morse.

\begin{lemma}
Let $\phi:\R^{2N}\ra\R$ be a real analytic function. Consider the sequence $(q_k)_{k=0}^{\infty}$ of stationary points of the process described above. If $q\in\R^{2N}$ is an  accumulation point of  the sequence $(q_k)$ then it is a critical point of $\phi$.
\label{bir}
\end{lemma}
\begin{proof}
Suppose that  $(q_k)$ is not ultimately constant so that every neighborhood of $q$ contains infinitely many distinct $q_k$'s. Either every neighborhood of $q$ contains infinitely many $q_k$'s of type $m$ for every $m=1,\ldots,N$ or not. In the former case by the continuity of $G$, we must have $G(q)=0$ which proves the lemma. Now we prove by contradiction that the latter case is not possible. Indeed, suppose  there is an open ball $U$ centered at $q$ which does not contain any point of the sequence of, say, type 1 and contains infinitely many points $q_{k_i-1}$ of type $N$. 
%\marginpar{\footnotesize Indices of  $\gamma_{k_i}$ do not go along with the rest. } 
%Moreover let $q'$ be another accumulation point of $(q_k)$ that contains infinitely many points $p_l$ of the sequence $(q_k)$ of type~1. 

We assumed that each curve $\gamma_{k_i}$, starting from the point $q_{k_i-1}$, leaves $U$ and converges to the point $q_{k_i}\not\in \bar{U}$. For any given $\epsilon>0$, there is a sufficiently large $K$ such that if $i>K$ we have %$\epsilon>\phi(q_{k_i-1})-\phi(q_{k_i})$. %Moreover, choosing $s$ sufficiently large, one can guarantee that from the point $p_l$ to the point $p_s$ the process visits a point $q_{k_i}$. 
%Then we have  
\begin{equation}
\epsilon>\phi(q_{k_i-1})-\phi(q_{k_i})>\int_{\gamma_{k_i}\cap U}\!\!\! |G^1_{k_i}|  > \frac{1}{2}\diam(U)\cdot \mbox{avg}(|G^1_{k_i}|)
\label{sifir}
\end{equation}
where avg$(|G^1_{k_i}|)$ is the average of $|G^1_{k_i}|$ over $\gamma_{k_i}\cap U$. 
Let $A=\bigcup_i \gamma_{k_i}$ and  $\gamma = (\mbox{closure}(A)-A)\cap U$. The set $\gamma$ can be considered as the set of limit points in $U$ of all point sequences $(p_l)$ with $p_l\in \gamma_{k_l}$,  $(l)$ being any increasing sequence of positive integers. Observe that $\gamma$ lies in the intersection of the  type-1 slice $R^1_q$ and the $\phi(q)$-level set. Moreover it follows from (\ref{sifir}) that  $G^1_q$ is zero on $\gamma$, i.e. all points of $\gamma$ are critical points of $G^1_{q}$. (Loosely speaking $\gamma$ is the {\it limit} of the arcs $\gamma_{k_i}\cap \bar{U}$ however we do not claim and do not need that $\gamma$ is a real analytic arc.)

Now since $\bar{\gamma}\subset \bar{U}$ is compact and all its points are critical points of $G^1_{q}$ and  $\phi(\bar{\gamma})=\phi(q)$, then one can choose a \L ojasiewicz neighborhood $W\subset R^1_q \cap U$ of $\gamma$  such that, assuming  $\phi(q)=0$, there are $c>0$ and $\mu\in[0,1)$ satisfying
$$|G^1_q(x)|> \frac{c}{2}|\phi(x)|^{\mu},$$
for all $x\in W$. (The triple $(W,c,\mu)$ is found as follows: first choose a \L ojasiewicz neighborhood $(W_x,c_x,\mu_x)$ of every point $x\in\gamma$, then take a finite subcollection $(W_l,c_l,\mu_l)$ covering $\bar{\gamma}$ and set $W=\cup W_l\cap U$, $c=\disp \min_l c_l$ and $\disp \mu=\max_l \mu_l$.) By the continuity of $G^1_u$ with respect to point $u$, one can extend $W$ to an open $V\subset U$ such that 
\begin{equation}
|G^1_u(x)|\geq c|\phi(x)|^{\mu},
\label{voja}
\end{equation}  
for all $u\in V$ and $x\in R^1_u$.
The set $V$ can be chosen smaller  (as a neighborhood of ${\gamma}$)  such that $\sup_V \phi <\delta$ where $\delta>0$ is as small as required. Note that
$\diam(U)\geq \diam(V)\geq \diam(\gamma)\geq\diam(U)/2$
% \marginpar{\footnotesize esitlik yanlis} 
and that, say,
\begin{equation}  
\mbox{ length } (\gamma_{k_i} \cap V) \geq \diam(V)/2.
\label{leng}
\end{equation}  
for large $i$'s. On the other hand using (\ref{voja}) we observe, as in the proof of \L ojasiewicz inequality (see e.g. \cite{kurd}), that there is a finite upper bound, for sufficiently large $i$ but independent from  $i$, for the length of $\gamma_{k_i}$ in $V$:
$$ \int_{\gamma_{k_i}(t) \cap V  \neq \emptyset} |\dot{\gamma_{k_i}}| dt
= \int_{0}^{t_{1}} |\dot{\gamma_{k_i}}(t)|dt \leq c_1(\phi(\gamma_{k_i}(t_{1})))^{1-\mu}<c\delta^{1-\mu}.$$ 
where $c_1=(c(1-\mu))^{-1}$.
Here $t_1$ is the moment when $\gamma_{k_i}$ leaves $U$.
This upper bound can be made arbitrarily small (that holds for larger $i$), in particular, much less than $\diam(V)$. This contradicts with (\ref{leng}). Therefore there must be infinitely many points of type 1 in $U$ and the proof by contradiction follows recursively.
\end{proof}

From the above proof, it follows  that sufficiently small neighborhoods of an accumulation point do not let trajectories out. This proves the following statement.
\begin{coro}
Let $\phi:\R^{2N}\ra\R$ be a real analytic function 
%and $q\in\R^{2N}$ be a critical point of $\phi$. Suppose that 
and $q$ an  accumulation point of  the sequence $(q_k)$ of stationary points of the process. Then $\lim_{k\ra+\infty} q_k=q$.
\label{iki}
\end{coro}

To finish the proof of the main theorem, we need the following
\begin{lemma}
Let $\phi:\tilde{\Omega}\ra\R$ be a primal navigation function on $\Omega\subset\tilde{\Omega}$. Then 
except when the initial point is  a saddle point,
convergence to a saddle point is impossible for a $C^{\omega}$-dense subset of primal navigation functions.
\label{uc}
\end{lemma}

\begin{proof}
Let $q=0$, $\phi(q)=0$, $s$ be the index of  $\phi$ at $q$ and $(U,f)$ be a real analytic Morse neighborhood of $q$,
i.e. $f:(\R^{dN},0)\ra(\R^{dN},0)$ such that  $g\doteq \phi\circ f^{-1}:(U,0)\ra (\R,0), (y_1,\ldots,y_{dN})\mapsto -y_1^2-\ldots-y_s^2+y_{s+1}^2+\ldots+y_{dN}^2$.

Now in this new setting,  $-\nabla g$ has the first $s$ components nonnegative. Therefore the first $s$ coordinates of the position cannot converge to 0 unless they are already 0. A necessary condition for that
is when $q_{k}\in U_0=U\cap (\{0\}\times \R^{dN-s})$ for some $k>0$. Furthermore in such a case, convergence is possible only when 
there are slices which contain integral curves in their intersection with $U_0$;
more precisely, for the slice $S$ with the projected vector  field $G_S$,  $U_0\cap S$ contains some integral curves of $G_S$. We will argue that such a case does not occur for a $C^{\omega}$-dense subset of primal navigation functions. 

First we perturb $\phi$ to a  primal navigation function $\psi$ with the same minimum point $p$ such that 
%the following local transversality condition is satified: 
the local stable manifold $U_q$ of each critical point $q$ of $\psi$ and the slices $R^m_x$ (for every $m$ and $x\in U_q$ sufficiently close to $q$)  intersect transversely.
It is straightforward to see that a 
sufficiently small generic real analytic coordinate change in a small neighborhood of $q$ achieves this locally.
% is determined by $d$ vectors in $\R^{dN}$. 
For example, a small generic rotation centered at the critical point would work. 
However we need further that this real analytic perturbation can be done globally, fixing  $p$.
We propose the following construction in case $\Omega$ is the closed unit ball. The construction
in the general case will follow since   $\Omega$ is analytically diffeomorphic to the closed unit ball.
Let $o$ be a point in the interior of $\Omega$ such that for each critical point $q$ of $\phi$
the line $oq$ is transverse and non-orthogonal  to $\bigcup_{1\leq m\leq N} R^m_q \cap U_q$,
i.e.  each intersection $R^m_q \cap U_q$ is non-tangent to $oq$ 
and has a radial (i.e. $oq$) component. 
%  and in a sufficiently small neighborhood of $q$.
Since these conditions are open, the set of such $o$'s is open and 
dense in $\Omega$. Below we take $o$ as the origin in the unit closed ball $\Omega$.
The construction below can be made for the general case  similarly.

Now we consider the real analytic diffeomorphism $h:\R^{dN}\ra\R^{dN}$, 
$h(x)=x+bx\frac{\sin ar^2}{r^2}$
where $r$ is the distance $|ox|$ and $a,b\in\R$. We require
$a$ satisfy $|op|^2=2\pi k/a$ for some $k\in\Z^+$.
% and  for each critical point $q$ there is some $l\in\Z^+$ such that  $0\leq |oq|-2\pi l/a<\delta<<\pi/2$.
%where $\delta$ is sufficiently small and to be determined below.
The radial perturbance $h$ fixes $p$ and moves the other critical points  slightly, the extent of which can be controlled by $b$. 
One can choose $a$ such that the radial perturbance of each slice at each critical point is nontrivial, which is required for transversality.
We also choose the domain $\tilde{\Omega}$ of  $\phi$
%as to have $\tilde{\Omega}-{\Omega}$ sufficiently small
so that the condition (ii)$'$ is satisfied in $\tilde{\Omega}-V$ where $V$ is a 
compact subset of $\Omega$. 
By making $b$ small, one can have $h$ arbitrarily $C^{\omega}$-close to id and $h(\tilde{\Omega})\supset\Omega$.
These choices guarantee that the function $\psi=\phi\circ h:h^{-1}(\Omega)\ra\R$ is defined
and is a real analytic Morse function with the same minimum as $\phi$; its critical points are
in one to one correspondence with, arbitrarily close to and of the same index  as those of $\phi$.
Moreover  the condition (ii)$'$ is satisfied  by $\psi$ on $\partial \Omega$ provided that $h$ is sufficiently $C^{\omega}$-close to id.
Thus we obtain a  primal navigation function satisfying the required transversality conditions.

Finally let us note that since the condition that $R^m_q\cap U_q$ contains integral curves 
of $G^m_{q}$ is a closed condition, the construction above gives the perturbation required
 for almost all $o$, infinitely many $a$ and sufficiently small $b$.
\end{proof}

%Note that without the condition $s\geq d$, there is nothing obvious against intersections at 1 or higher dimensions. \marginpar{\footnotesize nonoscillation?}

\subsection{Proof of Theorem~\ref{ana}} Since $\Omega$ is compact, the sequence $(q_k)$ has an accumulation point, which is a critical point of $\phi$ by Corollary~\ref{iki}. The proof follows from Lemma~\ref{uc} except the last sentence of the theorem. We assert the last sentence by giving explicit examples for each case in Section~\ref{ornek} below. \hfill $\Box$

\subsection{Proof of Theorem~\ref{ekstra}}
The proof will be immediate after modifying the definition of the {\em type of a point}: we say that $q_k$ has type $j$ if  $G(q_k)$ has the $j$-th component 0. Note that $q_k$ has type $j$ only if the previous slice of the process has type $\mathbf{n}_k$ such that $j\in \mathbf{n}_k$. 

Now we start as in Lemma~\ref{bir}. Suppose  that every neighborhood of $q$ contains infinitely many distinct sequence points. Either every neighborhood of $q$ contains infinitely many $q_k$'s of type $j$ for every $j=1,\ldots,M$ or not. In the former case by the continuity of $G$, we must have $G(q)=0$ so that we reach the conclusion of Lemma~\ref{bir}. In the latter case suppose  there is an open ball $U$ centered at $q$ which  contains a subsequence  $(q_{k_i-1})$ of points of type, say, 1. %For the sake of easiness of notation, let us denote by $\gamma_i$ the arcs of the process starting from $p_i$ and the type of slice containing $\gamma_i$ by $n_i$. 
Suppose infinitely many $\mathbf{n}_{k_i}$'s contain the index, say, 2. Denote the corresponding  point subsequence by   $(q_{l_i-1})$. Then the main body of Lemma~\ref{bir} proves that all type~2 sequence points $q_{l_i}$ are contained  in $U$. The sets $\mathbf{n}_{k_i}$ and $\mathbf{n}_{l_i+1}$ corresponding to the type~1 points $q_{k_i-1}$ and type~2 points $q_{l_i}$ respectively, either contain a different index than 1 and 2 infinitely many times, or ultimately the  indices that appear are nothing but 1 or 2. In the latter case the subsequence $(q_{l_i-1},q_{l_i})$ constitute a tail of $(q_k)$ so that this latter case presents a contradiction to the assumption of the theorem that all indices appear infinitely many times. Hence we are left with the former case and the claim of Lemma~\ref{bir} follows recursively as before. 

Proof of Lemma~\ref{uc} works exactly the same. \hfill $\Box$

\section{Further discussion and examples}
\subsection{Convergence analysis}
%However more is to be said to show that the accumulation point is the limit of the sequence $(q_k)$. Assuming $q$ is a critical point of $\phi$,  we will give a proof  of this
Considering the proof of Lemma~\ref{bir} above, more can be said about convergence to a critical point.  Through a more explicit convergence analysis one can control the lengths of the trajectories. 
This analysis shall reflect the nature of the process near a critical point.

Let $q$ be a critical point,  $\phi(q)=0$. 
We will bound from above the total length of the trajectories  once the sequence gets sufficiently close to $q$ in $U$. More precisely here is our claim.

\begin{prop}   
Let $(U=B_{r}(q),c,\mu)$ be a \L ojasiewicz neighborhood of $q$.
Suppose that for some $c'>1/(c(1-\mu))$ there is a $q_l\in U$ satisfying 
\begin{equation}
c'\phi^{1-\mu}(q_l)<r
\label{sipraym}
\end{equation}
and that $q_i\in U$ for $l\leq i\leq n$ for some $n>l$.
Then the total length of the trajectories $(\gamma_i)_{l<i\leq n}$ is less than $r$. 
\label{uzunluk}
\end{prop}

\begin{proof}
We use the idea in the proof of  \cite[Theorem~2.2]{ama}, where the problem considered 
involves  a fixed vector field, and a completely different set-up.
Set $\delta=1/(cc'(1-\mu))$ (observe $0<\delta<1$).
For each $l< i\leq n$ choose $\epsilon_i>0$ such that on the $\epsilon_i$-neighborhood  $B_i=B_{\epsilon_i}(q_i)\cap R^i_{q_{i}}$ on the slice $R^i_{q_{i}}$ (that contains the points $q_{i-1}$ and $q_i$ and the curve $\gamma_i$), the curve $\gamma_i$   satisfies the {\em angle condition}:
\begin{equation} \frac{d}{dt}\phi(\gamma_i (t))=\langle \nabla\phi(\gamma_i (t)), \dot\gamma_i (t)\rangle\leq -\delta|\nabla\phi(\gamma_i (t))|\cdot |\dot\gamma_i (t)|.
\label{del}
\end{equation}
Such $\epsilon_i$'s exist since at point $q_i$, $\nabla\phi(q_i)$ is perpendicular to the slice $R^i_{q_{i}}$.
Now if necessary, we make each  $\epsilon_i$ smaller so that the corresponding $B_i$ becomes a \L ojasiewicz neighborhood for the vector field $G^{m_i}_{q_i}$ too (here $q_i$ is of type $m_i$), i.e. there are $c_i>0$ and $\mu_i\in[0,1)$ such that for all $x\in B_i$, we have 
\begin{equation} |G^{m_i}_{q_i}(x)|\geq c_i|\phi(x)-\phi(q_i)|^{\mu_i}.
\label{ufak}
\end{equation}
Set $c'_i=1/(c_i\delta(1-\mu_i))$.  

Finally let $\Delta>0$ be given. Make $\epsilon_i$'s further smaller, if necessary, so that 
%the chosen $\epsilon_i$'s satisfy the condition that 
for every $a\in\{\mu,\mu_{l+1},\ldots,\mu_{n}\}$ and for every $x\in\bar{B_i}$:
\begin{equation} |\phi(x)^{1-a}-\phi(q_i)^{1-a}|<\frac{\Delta}{C(n-l)},
\label{maxC}
\end{equation}
where $C=\max\{c'_i \,|\, l<i\leq n\}$. This can be fullfilled since each $\phi^{1-a}$ is continuous.  
 Then 
%for $C'=\max\{c',c_i \,|\, l<i\leq n\}$ and $M=\max\{1-\mu,1-\mu_i \,|\, l<i\leq n\}$ 
we have the total length of the trajectories $\gamma_i$, ($l<i\leq n$):
$$ \ba{llll} L &= & \disp \sum_{i=l+1}^n  \int_{\gamma_i} |\dot{x}| &\\

&= &\disp  \sum_{i=l+1}^n  \left(  \int_0^{t_i}  |\dot{x}|    +  \int_{t_i}^{\infty} |\dot{x}|   \right), \mbox{ $t_i$ is when}  & \mbox{$\gamma_i$ enters $B_i$  the last time}\\

&\leq &\disp  \sum_{i=l+1}^n  c'\left(\phi(q_{i-1})^{1-\mu} - \phi(\gamma_i(t_i))^{1-\mu}\right)  & \mbox{using (\ref{oja}), (\ref{del}), (\ref{ufak})}\\
&  & + \disp \sum_{i=l+1}^n  c'_i \left(\phi(\gamma_i(t_i))^{1-\mu_i} - \phi(q_{i})^{1-\mu_i}\right), &  \mbox{ and \cite[Equation~2.6]{ama}} \\

&\leq & \disp c' \phi(q_l)^{1-\mu} - c'\phi(\gamma_n(t_n))^{1-\mu} & \\
& & + \disp  \sum_{i=l+1}^{n-1}  c' \left(\phi(q_i)^{1-\mu} - \phi(\gamma_i(t_i))^{1-\mu}\right) & \mbox{(these terms are negative)} \\
& & + \disp \sum_{i=l+1}^n  c'_i \left(\phi(\gamma_i(t_i))^{1-\mu_i} - \phi(q_{i})^{1-\mu_i}\right) & \\
%\mbox{by choice of $C'$ and $M$}\\

& < & c' \phi(q_l)^{1-\mu} & \mbox{other terms are negative} \\
& & +  \disp \sum_{i=l+1}^n  c'_i \frac{\Delta}{C(n-l)}, & (\ref{maxC}) \\

& < & r + \Delta, &  \mbox{(\ref{sipraym}) and $c'_i\leq C$.}
\ea
$$
Since $\Delta$ was arbitrary, we deduce $L\leq r$. 
\end{proof}

We close this section with a consequence of the proposition.
\begin{coro}
With the hypotheses of Proposition~\ref{uzunluk} suppose $q_i\in U$ for all $i\geq l$ (so that $\lim q_i=q$). Then the total length of the trajectories $(\gamma_i)_{i>l}$ is less than $r$. 
\end{coro}
 
\subsection{Examples}
\label{ornek}
\begin{figure}[h]
   \begin{center}
\resizebox{10cm}{!}
{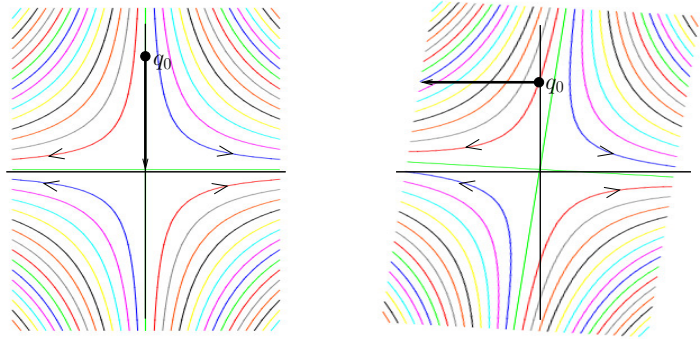}
\caption{Integral curves of a function. (a)The process hits $q$ after finitely many steps; (b) A small perturbation spoils the convergence; by definition the process starts with a type-1 step.}
       \label{ornek1}
      \end{center}
\end{figure}
A case where the process hits a critical point $q$ after a finite number of steps is easy to cook up locally (see Figure~\ref{ornek1}(a)). We have seen that this case is not generic. In fact Figure~\ref{ornek1}(b) illustrates a small generic real analytic local perturbation of (a) where there is no initial $q_0$ for which the process hits 0. 

Similarly, Figure~\ref{sisman} depicts a situation where convergence to a saddle point is provided
by a set of initial points  of positive area. Note that this case occurs only if a slice lies in an stable manifold and the slice contains an integral curve of the projected gradient system.
\begin{figure}[t]
   \begin{center}
\resizebox{7cm}{!}
{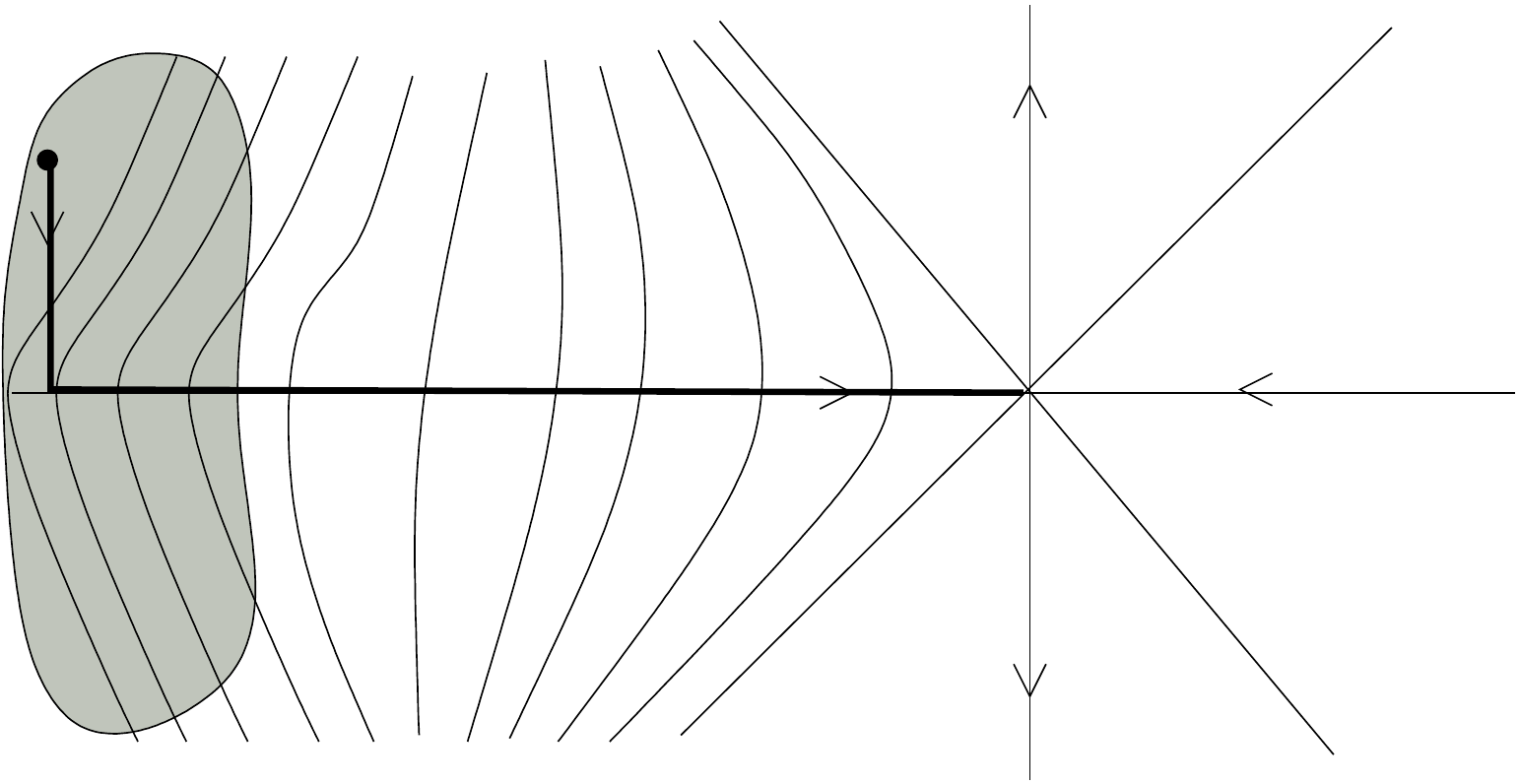}
\caption{Level curves of a function. The initial points in the shaded region start a process converging to the saddle point.}
       \label{sisman}
      \end{center}
\end{figure}

One might ask to herself if infinitely many steps  is ever required for convergence to a critical point.  The answer is affirmative and an example for a minimum point is depicted in Figure~\ref{ornek2}.  
%Such a convergence would violate the nonoscillation property. 

\begin{figure}[t]
   \begin{center}
\resizebox{7cm}{!}
{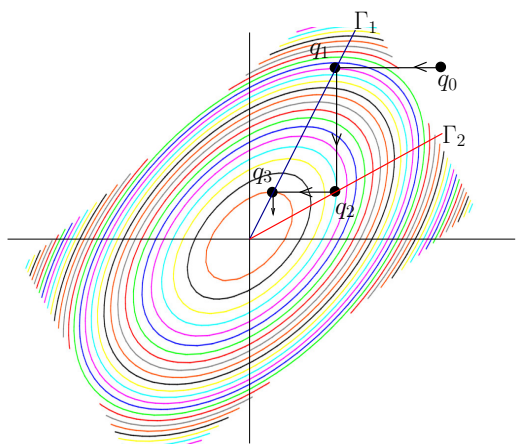}
\caption{$a(x+y)^2+b(x-y)^2=1, a\neq b$. Infinite steps needed for convergence to 0.}
       \label{ornek2} 
      \end{center}
\end{figure}

It is not difficult either to imagine a case where  infinitely many steps  is needed for convergence to a saddle point. Let $d=1$, $N=3$, $s=1$, $b=0$ and $c=2$, and the type-1, -2, -3 slices be lines  parallel to $x$-, $y$- and $z$-axes respectively. We consider the Morse function
$f(x,y,z)=(x-z)^2+2(x+z)^2-2y(3x+z)$ which has a unique critical point at 0. 
Observe that the curve $\beta=\Gamma_1\cap \Gamma_2$ lies in $0\times \R^2$. We start from an initial point $q_0\in (0\times \R^2) - \alpha$, where $\alpha=\Lambda_1\cap (0\times \R^2)$. The type-1 steps  terminate at points on the curve $\alpha$, which are  critical points for the type-2 steps as well.   Thus the process makes zigzags between $\alpha$ and $\beta$ via type-1 and type-3 steps.
% (see Figure~\ref{ornek3}).
The behaviour on  $0\times \R^2$ is exactly the same as in Figure~\ref{ornek2}.

\end{document}

%% file: pert.pdf_t
\begin{picture}(0,0)%
\includegraphics{pert.pdf}%
\end{picture}%
\setlength{\unitlength}{4144sp}%
\begingroup\makeatletter\ifx\SetFigFont\undefined%
\gdef\SetFigFont#1#2#3#4#5{%
  \reset@font\fontsize{#1}{#2pt}%
  \fontfamily{#3}\fontseries{#4}\fontshape{#5}%
  \selectfont}%
\fi\endgroup%
\begin{picture}(8685,4410)(1441,-5011)
\put(8041,-1786){\makebox(0,0)[lb]{\smash{{\SetFigFont{17}{20.4}{\rmdefault}{\mddefault}{\updefault}{\color[rgb]{0,0,0}$q_0$}%
}}}}
\put(3511,-1516){\makebox(0,0)[lb]{\smash{{\SetFigFont{17}{20.4}{\rmdefault}{\mddefault}{\updefault}{\color[rgb]{0,0,0}$q_0$}%
}}}}
\end{picture}%

%% file: sisman.pdf_t
\begin{picture}(0,0)%
\includegraphics{sisman.pdf}%
\end{picture}%
\setlength{\unitlength}{4144sp}%
\begingroup\makeatletter\ifx\SetFigFont\undefined%
\gdef\SetFigFont#1#2#3#4#5{%
  \reset@font\fontsize{#1}{#2pt}%
  \fontfamily{#3}\fontseries{#4}\fontshape{#5}%
  \selectfont}%
\fi\endgroup%
\begin{picture}(7042,3624)(171,-3223)
\end{picture}%

%% file: elips1.pdf_t
\begin{picture}(0,0)%
\includegraphics{elips1.pdf}%
\end{picture}%
\setlength{\unitlength}{4144sp}%
\begingroup\makeatletter\ifx\SetFigFont\undefined%
\gdef\SetFigFont#1#2#3#4#5{%
  \reset@font\fontsize{#1}{#2pt}%
  \fontfamily{#3}\fontseries{#4}\fontshape{#5}%
  \selectfont}%
\fi\endgroup%
\begin{picture}(6795,5633)(4861,-5172)
\put(8176,-1741){\makebox(0,0)[lb]{\smash{{\SetFigFont{20}{24.0}{\rmdefault}{\mddefault}{\updefault}{\color[rgb]{0,0,0}$q_3$}%
}}}}
\put(9241,-2191){\makebox(0,0)[lb]{\smash{{\SetFigFont{20}{24.0}{\rmdefault}{\mddefault}{\updefault}{\color[rgb]{0,0,0}$q_2$}%
}}}}
\put(10561,-541){\makebox(0,0)[lb]{\smash{{\SetFigFont{20}{24.0}{\rmdefault}{\mddefault}{\updefault}{\color[rgb]{0,0,0}$q_0$}%
}}}}
\put(8911,-136){\makebox(0,0)[lb]{\smash{{\SetFigFont{20}{24.0}{\rmdefault}{\mddefault}{\updefault}{\color[rgb]{0,0,0}$q_1$}%
}}}}
\put(9481,194){\makebox(0,0)[lb]{\smash{{\SetFigFont{20}{24.0}{\rmdefault}{\mddefault}{\updefault}{\color[rgb]{0,0,0}$\Gamma_1$}%
}}}}
\put(10621,-1276){\makebox(0,0)[lb]{\smash{{\SetFigFont{20}{24.0}{\rmdefault}{\mddefault}{\updefault}{\color[rgb]{0,0,0}$\Gamma_2$}%
}}}}
\end{picture}%

%% file: robos.bbl
\begin{thebibliography}{99}


\bibitem{ama} Absil P. A.;  Mahony~ R.; Andrews B. Convergence of iterates of descent methods  for analytic cost functions. SIAM J. Optim. Vol. 16, No. 2, pp. 531-547, 2005.

\bibitem{bar} Baryshnikov, Y.; Bubenik, P.; Kahle, M.
Min-type Morse theory for configuration spaces of hard spheres.
Int. Math. Res. Not. IMRN 2014, no. 9, 2577–2592. 


\bibitem{bozkodrob2001}
Bozma, H.I., D.E.Koditschek. Assembly as a Noncooperative Game of its Pieces: Analysis of 1D Sphere Assemblies, Robotica, pp: 93-108, Jan 2001.

\bibitem{far} Farber,  M. Topological complexity of motion planning. Discrete and Computational Geometry,Vol. 29, No. 2, 211-221, 2003

\bibitem{fco} Cohen, D. C.; Farber, M. Topological complexity of collision-free motion planning on surfaces. Compos. Math. 147 (2011), no. 2, 649-660.

\bibitem{loj0} \L ojasiewicz, S.
Une propri\'{e}t\'{e} topologique des sous-ensembles analytiques r\'{e}els.
Colloques Internationaux du C.N.R.S. no 117, les equations aux d\'{e}riv\'{e}es partielles,
Paris 25 -- 30 juin (1962), p. 87-89

\bibitem{loj} \L ojasiewicz, S.
Sur les trajectoires du gradient d'une fonction analytique. Geometry seminars, 1982-1983 (Bologna, 1982/1983), 115--117, Univ. Stud. Bologna, Bologna, 1984. 

\bibitem{boz} Karag\"{o}z C. S.; Bozma H. I.; Koditschek Daniel E. Coordinated Navigation of Multiple Independent Disk-Shaped Robots. IEEE Trans. on Robotics 30 (2014), no. 6, 1289--1304.

\bibitem{ko1} Koditschek, D. E.; Rimon, E. Robot navigation functions on manifolds with boundary. Adv. in Appl. Math. 11 (1990), no. 4, 412--442.

\bibitem{ko2}  Rimon, E.; Koditschek, D. E. The construction of analytic diffeomorphisms for exact robot navigation on star worlds. Trans. Amer. Math. Soc. 327 (1991), no. 1, 71--116.

\bibitem{kurd} Kurdyka, K.; Mostowski, T.; Parusi\'{n}ski, A.
Proof of the gradient conjecture of R. Thom. 
Ann. of Math. (2) 152 (2000), no. 3, 763--792.

\bibitem{lat}  Latombe, J.-C. Robot motion planning (Kluwer, Dordrecht, 1991).

\bibitem{mil} Milnor, J. Morse theory. Based on lecture notes by M. Spivak and R. Wells. Annals of Mathematics Studies, No. 51 Princeton University Press, Princeton, N.J. 1963.





\end{thebibliography}
